\documentclass[12pt,reqno]{amsart}

\usepackage{amsfonts,amsthm,amsmath,microtype}
\newcommand{\numberset}{\mathbb} 
\newcommand{\R}{\numberset{R}} 

\usepackage[a4paper,top=3cm,bottom=3cm,left=3cm,right=3cm,bindingoffset=0mm]{geometry}

\usepackage{verbatim}
\usepackage{setspace}

\DeclareMathOperator{\Tr}{Tr}

\makeatletter

\numberwithin{equation}{section}
\newtheorem{thm}{\indent\bf Theorem}[section]
\newtheorem{lemma} [thm] {\indent\bf Lemma}

\newtheorem{prop}[thm]{\indent\bf Proposition}

\begin{document}
\title[Weighted multipolar Hardy inequalities]
{Weighted Hardy inequalities and
Ornstein-Uhlenbeck type operators perturbed by multipolar inverse square potentials}

\author[A. Canale]{Anna Canale}
\address{Dipartimento di Ingegneria dell'Informazione ed Elettrica e Matematica Applicata (Diem), 
Universit\'a degli Studi di Salerno, Via Giovanni Paolo II, 132, 84084 Fisciano
(Sa), Italy.}
\email{acanale@unisa.it}
\author[F.Pappalardo]{Francesco Pappalardo}
\address{Dipartimento di Matematica e Applicazioni \textquotedblleft Renato  Caccioppoli\textquotedblright,
Universit\'a degli Studi di Napoli Federico II, Complesso Monte S.  Angelo, Via Cinthia, 80126 Napoli, Italy.}
\email{francesco.pappalardo@unina.it}

\thanks{The authors are members of the Gruppo Nazionale per l'Analisi Matematica, la Probabilit\'a e le loro Applicazioni 
(GNAMPA) of the Istituto Nazionale di Alta Matematica (INdAM)}


\begin{abstract}
We give necessary and sufficient conditions for the
existence of weak solutions of a parabolic problem corresponding to the Kolmogorov operators perturbed 
by a multipolar inverse square potential
\begin{equation*}
Lu+Vu=\left(\Delta u+\frac{\nabla \mu}{\mu}\cdot \nabla u\right)+
\sum_{i=1}^n \frac{c}{|x-a_i|^2}u, \quad x\in \R^N,\>c>0,\> a_1, \dots, a_n\in \R^N,
\end{equation*}
defined on smooth functions where $\mu$ in the drift term is a probability density on $R^N$.
To this aim we state a weighted Hardy inequality 
\begin{equation*}
c\sum_{i=1}^n\int_{\R^N}\frac{\varphi^2}{|x-a_i|^2}\,d\mu
\leq\int_{\R^N}|\nabla \varphi |^2d\mu+
K\int_{\R^N} \varphi^2d\mu,\quad \varphi\in H^1_\mu,\quad c\le c_o,
\end{equation*}
where $c_o=c_o(N):=\left( \frac{N-2}{2} \right)^{2}$,
with respect to the Gaussian probability measure $d\mu=\mu(x)dx$
which is the unique invariant measure for Ornstein-Uhlenbeck type operators.
We state the optimality of the constant  $c_o$ and, then,
the nonexistence of positive exponentially bounded solutions to the parabolic problem. 
\end{abstract}
\maketitle

\section{Introduction}
The paper deals with a class of Kolmogorov operators 
\begin{equation}\label{L}
Lu=\Delta u+\frac{\nabla \mu}{\mu}\cdot\nabla u,
\end{equation}
 perturbed by a multipolar inverse square potential
\begin{equation}\label{V}
V(x)=\sum_{i=1}^n \frac{c}{|x-a_i|^2},\quad x\in \R^N,\quad c>0,\quad a_1, \dots, a_n\in \R^N,
\end{equation}
defined on smooth functions and $\mu$ is a probability density on $\R^N$.

From the mathematical point of view, the interest in inverse square potentials of type $V\sim\frac{c}{|x|^2}$
relies in the criticality: they have the same homogeneity as the Laplacian and do not belong to the Kato's class,
then they cannot be regarded as a lower order perturbation term.
Furthermore the study of such singular potentials is motived by applications  to many fields, for example 
in many physical contexts as molecular physics
\cite{ Levy}, 
quantum cosmology 
(see e.g. \cite{ BerestyckiEsteban}),
quantum mechanics \cite{BarasGoldstein} and
combustion models 
 \cite{ Gelfand}.

Multipolar potentials are associated with the interaction of a finite number of electric dipoles 
as, for example, in molecular systems consisting of $n$ nuclei of unit charge located in
a finite number of points $a_1, \dots, a_n$ and of $n$ electrons. The Hartree-Fock model
describes these  systems (see \cite{Catto}).

It is well known that if $L=\Delta$ and $V\le \frac{c}{|x|^{2-\epsilon}},\,c>0,\,\epsilon>0$, 
then the corresponding initial value problem is well-posed. But for $\varepsilon=0$ the problem
may not have positive solution.
In \cite{BarasGoldstein} Baras and Goldstein showed that 
the evolution problem associated to $\Delta+V$ admits a unique positive solution
if $c\leq c_o(N):=\left( \frac{N-2}{2} \right)^{2}$ and no positive solutions exist if $c>c_o(N)$
(see also \cite{CabreMartel} for a different approach involving the Hardy inequality).
When it exists, the solution is
exponentially bounded, on the contrary,  if $c>c_o(N)$, there is the so called instantaneous blowup phenomena.

A similar behaviour was obtained in \cite{GGR} with the potential $V=\frac{c}{|x|^2}$  and replacing 
the Laplacian by the Kolmogorov operator $L$ . See also \cite{CGRT} where 
the hypotheses on $\mu$ allow the drift term to be of the type 
$\frac{\nabla \mu}{\mu}=-|x|^{m-2}\,x$, $m>0$.

In this paper we consider the generalized Ornstein-Uhlenbeck operator in $\R^N$
\begin{equation}\label{OU}
Lu=\Delta u -\sum_{i=1}^n A(x-a_i)\cdot \nabla u,
\end{equation}
where $A$ is a positive definite real Hermitian $N\times N$-matrix, and the associated evolution problem
$$
(P)\quad \left\{\begin{array}{ll}
\partial_tu(x,t)=Lu(x,t)+V(x)u(x,t),\quad \,x\in {\mathbb R}^N, t>0,\\
u(\cdot ,0)=u_0\geq 0\in L_\mu^2,
\end{array}
\right. $$
with the multipolar singular potential  $V$ defined in (\ref{V}) and $L_\mu^2$ a suitable weighted space.

We state existence and nonexistence results in the case of the generalized Ornstein-Uhlenbeck operator
using the relationship between the weak solution of $(P)$
and the {\it bottom of the spectrum} of the operator $-(L+V)$
\begin{equation*}
\lambda_1(L+V):=\inf_{\varphi \in H^1_\mu\setminus \{0\}}
\left(\frac{\int_{{\mathbb R}^N}|\nabla \varphi |^2\,d\mu
-\int_{{\mathbb R}^N}V\varphi^2\,d\mu}{\int_{{\mathbb R}^N}\varphi^2\,d\mu}
\right).
\end{equation*}
When $\mu=1$ Cabr\'e and Martel in \cite{CabreMartel}  
showed that the boundedness  of $\lambda_1(\Delta+V)$, 
$0\le V\in L_{loc}^1({\mathbb R}^N)$, is a necessary 
and sufficient condition for the existence of positive exponentially bounded in time
solutions to the associated initial value problem. 
Later in \cite{GGR} the authors extended  such a result to the case of Kolmogorov operators.

The estimate of the bottom of the spectrum $\lambda_1(L+V)$ is equivalent to the weighted Hardy inequality 
with $V(x)=\sum_{i=1}^n \frac{c}{|x-a_i|^2}$, $c\le c_o(N)$,  
\begin{equation}\label{whi}
\int_{\R^N}V\,\varphi^2\,d\mu
\leq\int_{\R^N}|\nabla \varphi |^2d\mu+
K\int_{\R^N} \varphi^2d\mu,\quad \varphi\in H^1_\mu,
\end{equation}
and the sharpness of the best  possible constant. As we will see in the next Section, $H^1_\mu$ denotes
an appropriate weighted Sobolev space.

Then the existence of positive solutions to $(P)$
is related to the Hardy inequality (\ref{whi})  and the nonexistence 
is due to the optimality of the constant $c_o$. 

Our results about Hardy-type inequalities (\ref{whi}) (see Theorem \ref{wH} and Theorem \ref{c_o su n} in Section 3)
fits into the context of the so called {\sl multipolar Hardy inequalities}.

When $\mu=1$ and ${\mathcal L}$ is the Schr\"odinger operator
$$
{\mathcal L}=-\Delta-\sum_{i=1}^n\frac{c_i^+}{|x-a_i|^2},
$$
where $n\ge2$, $c_i\in \R$, $c_i^+=\max\{c_i,0\}$, for any $i\in \{1,\dots, n\}$, Felli, Marchini and Terracini in
\cite{FelliMarchiniTerracini} proved that the associated quadratic form
$$
Q(\varphi):=\int_{\R^N}|\nabla \varphi |^2\,dx
-\sum_{i=1}^n c_i\int_{{\mathbb R}^N}\frac{\varphi^2}{|x-a_i|^2}\,dx
$$
is positive if $\sum_{i=1}^nc_i^+<\frac{(N-2)^2}{4}$, conversely if
$\sum_{i=1}^nc_i^+>\frac{(N-2)^2}{4}$ there exists a configuration of poles such that $Q$ is not positive.
Later Bosi, Dolbeaut and Esteban in \cite{BDE} proved that for any $c\in\left(0,\frac{(N-2)^2}{4}\right]$ there exists 
a positive constant such that (\ref{whi}) holds.
Recently Cazacu and Zuazua in \cite{CazacuZuazua}, improving a result stated in 
\cite{BDE}, obtained the inequality (\ref{whi}) with $K=0$ and 
$V= c\sum_{1\le i<j\le n}\frac{|a_i-a_j|^2}{ |x-a_i|^2|x-a_j|^2}$.

As far as we know there are no results in the literature about the weighted multipolar Hardy inequalities.

In this paper we are motivated to consider the Gaussian measure
 $d\mu(x)=\mu(x)dx=Ce^{-\frac{1}{2}\sum_{i=1}^n\langle A(x-a_i),(x-a_i)\rangle}dx$, with $C$ normalization constant,
which is the unique invariant measure for the Ornstein-Uhlenbeck type operator (\ref{OU})
whose drift term is unbounded at infinity. 

In Section 3 we will prove the inequality (\ref{whi}) in a direct way starting from the result obtained in \cite{BDE}
with the Lebesgue measure
and exploiting a suitable bound satisfied by the function $\mu$. 
Furthermore we will state  the optimality of the constant $c_o$ in (\ref{whi}).

Afterwards, in Section 4,  we will give a proof 
of the inequality through the so called \textit{IMS}  (Ismaligov, Morgan, Morgan-Simon, Sigal), method 
based on a suitable partition of unity in $\R^N$, reasoning as in \cite{BDE}. 
To this aim we need to use a Hardy inequality in the case $n=1$
which we will  prove. Indeed in the IMS method a fundamental tool is an estimate with a single pole
which allows us to reach the optimal constant $c_o(N)$ in the inequality.

In Section 5 we will state one of the main results, Theorem \ref{Hardy e CM theor}, 
which put together weighted Hardy inequality and Theorem \ref{theor as CM}
sating an existence and nonexistence result.
Furthermore, using the bilinear form associated to the operator $-(L+V)$, we will state the generation 
of an analytic $C_0$-semigroup and the positivity of the solution arguing as in \cite{ArendtGoldstein}.

\bigskip\bigskip

\section{Notation and preliminary results}
\bigskip

Let us consider Kolmogorov operators $L$ defined in (\ref{L}) and
the functions $\mu\in C_{loc}^{1,\alpha}\left( \R^N \right)$
for some $\alpha\in (0,1)$, $\mu(x)>0$ for all $x\in \R^N$.

It is known that the operator $L$ with domain
$$D_{max}(L)=\lbrace u\in C_b(\R^N)\cap W_{loc}^{2,p}(\R^N) \text{ for all }  1<p<\infty, 
Lu\in C_b(\R^N)\rbrace$$
is the weak generator of a not necessarily $C_0$-semigroup $\lbrace T(t)\rbrace_{t\ge 0}$ in $C_b(\R^N)$. 
Since $\int_{\R^N}Lu\,d\mu=0$ for any $u\in C_c^{\infty}(\R^N)$,
where $d\mu=\mu (x)dx$,
then $d\mu$ is the invariant measure for $\lbrace T(t)\rbrace_{t\ge 0}$ in $C_b(\R^N)$. 
So we can extend it to a positive preserving and analytic $C_0$-semigroup on $L^2_\mu:=L^2(\R^N,d\mu)$, 
whose generator is still denoted by $L$. 

Furthermore we denote by $H^1_\mu$ be the set of all the functions $f\in L^2_\mu$ having distributional derivative
 $\nabla f$ in $( L^2_\mu)^N$. 

We recall the following proposition 
 (see \cite[Chapter 8]{BertoldiLorenzi} for more details).
\begin{prop}
The following assertions hold:
\begin{itemize}
\item[i)] $C_c^\infty ({\mathbb R}^N)$ is a core for $L$ in $L^2_\mu$;
\item[ii)] $D(L)$ is continuously and densely embedded in $H^1_\mu$; 
\item[iii)] $\int_{{\mathbb R}^N}\nabla u\cdot \nabla
v\,d\mu=-\int_{{\mathbb R}^N}(Lu)v\,d\mu ,\quad u\in D(L),\,v\in H^1_\mu$;
\item[iv)] for any $t>0$, $T(t)L^2_\mu\subset H^1_\mu$.
\end{itemize}
\end{prop}
From i) and ii) follows that $C_c^\infty ({\mathbb R}^N)$ is densely embedded in $H^1_\mu$.
Then we can regard $H^1_\mu$ also as the completion of $C_{c}^{\infty}(\R^N)$ in the norm 
$$\|u\|_{H^1_\mu}^2 := \|u\|_{L^2_\mu}^2 + \|\nabla u\|_{L^2_\mu}^2.$$
The operator $L$ can also be defined via the bilinear form
\begin{equation}\label{a}
a_\mu(u,v)=\int _{\R^N}\nabla u\cdot \nabla v \,d\mu
\end{equation}
on $H^1_\mu$.
This is immediately clear by integrating by parts in (\ref{a}). Indeed
\[
a_\mu(u,v)=-\int_{\R^N} Luv \,d\mu,
\qquad u,v\in C_c^{\infty}(\R^N).
\]
Let us recall the problem
$$
\left(P\right)\quad
\left\lbrace
\begin{array}{ll}
\partial_t u(x,t)=Lu(x,t)+V(x)u(x,t), & t>0, x\in \R^N, N\ge 3,
\\
u(\cdot,t)=u_0\in L^2_\mu,
\end{array}\right.
$$
where $L$ is as in (\ref{L}).
We say that $u$ is a weak solution to ($P$) if, for each $T, R>0 $, we have
$$u\in C(\left[ 0, T \right] , L^2_\mu ), \quad Vu\in L^1(B_R \times \left( 0,T\right) , d\mu dt )$$
and
$$\int_0^T\int_{\R^N}u(-\partial_t\phi - L\phi )\,d\mu dt -\int_{\R^N}u_0\phi(\cdot ,0)\,d\mu =
\int_0^T\int_{\R^N}Vu\phi \, d\mu dt$$
for all $\phi \in W_2^{2,1}(\R^N \times \left[ 0,T\right])$ having compact support with $\phi(\cdot ,T)=0$, 
where $B_R$ denotes the open ball of $\R^N$ of radius $R$ centered at $0$.\\

For any $\Omega\subset \R^N$, $ W_2^{2,1}(\Omega\times (0,T)) $ is the parabolic Sobolev space 
of the functions $u\in L^2(\Omega \times (0,T)) $ having weak space derivatives $D_x^{\alpha}
u\in L^2(\Omega \times (0,T))$ for $|\alpha |\le 2$ and weak time derivative
$\partial_t u \in L^2(\Omega \times (0,T))$ equipped with the norm 

$$ \|u\|_{W_2^{2,1}(\Omega\times (0,T))}:= \Biggl( 
\|u\|_{L^2(\Omega \times (0,T))}^2 + \|\partial_t u\|_{L^2(\Omega \times (0,T))}^2 +
 \sum_{1\le |\alpha |\le 2} \|D^{\alpha}u\|_{L^2(\Omega \times (0,T))}^2
\Biggr)^{\frac{1}{2}}. $$
\medskip 

Now we can state the following result.

\begin{thm}\label{theor as CM}
Assume $0<\mu \in C^{1,\alpha}_{loc}({\mathbb R}^N)$
is a probability density on ${\mathbb R}^N$ and
$0\le V\in L_{loc}^1({\mathbb R}^N)$. Then the following assertion hold:
\begin{itemize}
\item[(i)] If $\lambda_1(L+V)>-\infty$, then there exists a
positive weak solution $u\in C([0,\infty),L^2_\mu)$ of $(P)$ satisfying
\begin{equation}\label{eq 1}
\|u(t)\|_{L^2_\mu}\le Me^{\omega t}\|u_0\|_{L^2_\mu},\quad t\ge0
\end{equation}
for some constants $M\ge 1$ and $\omega \in {\mathbb R}$.
 \item[(ii)] If
$\lambda_1(L+V)=-\infty$, then for any $0\le u_0\in
L^2_\mu\setminus \{0\},$ there is no positive weak solution of $(P)$
satisfying \eqref{eq 1}.
\end{itemize}
\end{thm}

The proof of the Theorem is based on Cabr\'e-Martel's idea in \cite{CabreMartel} 
and it was proved in \cite {GGR} for
functions $\mu$ belonging to $C_{loc}^{1,\alpha}(\R^N)$. The proof 
relies on certain properties of the operator $L$ 
and its corresponding semigroup $T(t)$ in $L^2_\mu$.
Furthermore the strict positivity on compact sets of $T(t)u_0$, if $0\leq u_0\in L^2_\mu
\setminus \{0\}$ is required.

\bigskip\bigskip

\section{Weighted Hardy inequality and optimality of the constant}
\bigskip

Let us consider the following Gaussian measure
\begin{equation}\label{mu}
d\mu =\mu(x)dx= C\, e^{-\frac{1}{2}\sum_{i=1}^{n}\langle A(x-a_i),x-a_i\rangle}\, dx
\end{equation}
with
\begin{equation}\label{C}
C=\left(\int_{\R^N} e^{-\frac{1}{2}\sum_{i=1}^{n}\langle A(x-a_i),x-a_i\rangle}\, dx \right)^{-1}
\end{equation}
and $A$  positive definite real Hermitian $N\times N$-matrix,
which is the unique invariant probability measure for Ornstein-Uhlenbeck type operators
$$
 Lu=\Delta u -\sum_{i=1}^n A(x-a_i)\cdot \nabla u. 
$$
So the operator $L$, with domain 
$H^2_\mu:=\{u\in H^1_\mu: D_ku\in H^1_\mu\rbrace$, 
 generates an analytic semigroup  $\lbrace T(t)\rbrace_{t\ge 0}$ on $L^2_\mu$ (cf \cite{MPRS}).

The operators we consider are perturbed by the multipolar inverse square potential   
\begin{equation}\label{Vn}
V(x)=\sum_{i=1}^n \frac{c}{|x-a_i|^2}=c\,V_n,
\end{equation}
where $x \in \R^N$, $c> 0$, $a_i\in \R^N$, $ i=1,\dots,n$.

We state the following weighted Hardy inequality.

\begin{thm}\label{wH}
Assume $N\ge 3$, $n\ge 2$ and $A$ a positive definite real Hermitian $N\times N$-matrix. 
Let $r_0=\min_{i\neq j} |a_i-a_j|/2$, $i,j=1,\dots,n$ and $k\in [0, \pi^2)$. Then we get

\begin{equation}\label{gaussian hardy}
\begin{split}
c\int_{{\R}^N}\sum_{i=1}^n\frac{\varphi^2 }{|x-a_i|^2}\, d\mu&\le  
\int_{{\R}^N} |\nabla\varphi|^2 \, d\mu 
\\&+
 \left[\frac{k+(n+1)c}{ r_0^2}+\frac{n}{2}\Tr A\right] \int_{\R^N}\varphi^2 \, d\mu 
\end{split}
\end{equation}
for all $\varphi \in H^1_\mu$, where $c\in (0, c_o]$ with $c_o=c_o(N):=\left(\frac{N-2}{2}\right)^2$ 
optimal constant.
\end{thm}

\bigskip

\begin{proof}
$$
$$

\noindent {\sl Step 1}$\;$({\it Inequality})$\quad$
\medskip

By density we can consider functions $\varphi\in C_c^{\infty}(\R^N)$.

The starting point is the following inequality, stated by Bosi, Dolbeault and Esteban in \cite[Theorem 1]{BDE} :

\begin{equation}\label{Esteban}
\begin{split}
c\int_{{\R}^N}\sum_{i=1}^n\frac{\varphi^2 }{|x-a_i|^2}\, dx\le  
\int_{{\R}^N} |\nabla\varphi|^2 \, dx +
 \left[\frac{k+(n+1)c}{ r_0^2}\right] \int_{\R^N}\varphi^2 \, dx 
\end{split}
\end{equation}
for all $\varphi \in H^1(\R^N)$, with $n\ge 2$, $k\in \left[ 0, \pi^2 \right)$ 
and $c\in (0,c_o]$.
The proof of  (\ref{Esteban}) is based on IMS truncation method.
In the Section \ref{via IMS} we will prove the weighted version of the inequality  (\ref{Esteban})  
reasoning as in \cite[Theorem 1]{BDE} .

Now we state the weighted version of this result in a direct way. 

Indeed, applying (\ref{Esteban})  to the function  $\varphi \sqrt{\mu}$, we have

\begin{equation*}
c\int_{{\R}^N}\sum_{i=1}^n\frac{\varphi^2 }{|x-a_i|^2}\, d\mu \le
\int_{{\R}^N} |\nabla\left(\varphi \sqrt{\mu} \right)|^2 \, dx +
\left[\frac{k+(n+1)c}{ r_0^2}\right]\int_{{\R}^N} \varphi^2 \, d\mu.
\end{equation*}
By means the easy calculation

\begin{equation*}
\begin{split}
 \int_{{\R}^N} |\nabla\left(\varphi \sqrt{\mu} \right)|^2 \, dx &=
\int_{{\R}^N} \left| (\nabla\varphi)\sqrt{\mu}+\varphi \frac{\nabla\mu}{2\sqrt{\mu}}\right|^2 \, d\mu\\
&=
 \int_{{\R}^N} |\nabla\varphi|^2 \, d\mu+ \int_{{\R}^N} \left[\frac{1}{4}
\left|\frac{\nabla\mu}{\mu} \right|^2-\frac{1}{2}\frac{\Delta\mu}{\mu} \right]\varphi^2 \, d\mu .
\end{split}
\end{equation*}
and observing that we can estimate the last integral above taking into account that

\begin{equation}\label{bound Tr}
\begin{split}
\frac{1}{4}\left| \frac{\nabla\mu}{\mu} \right|^2 -\frac{1}{2}\frac{\Delta\mu}{\mu} &=
\frac{1}{4}\left|\sum_{j=1}^{n}A(x-a_j) \right|^2+
\\&
-\frac{1}{2}\left[ -n \Tr A+\left|\sum_{j=1}^{n}A(x-a_j) \right|^2\right] 
\leq \frac{n}{2}\Tr A
\end{split}
\end{equation}
we get the result.

\bigskip

\noindent {\sl Step 2}$\;$({\it Optimality})$\quad$
\medskip

To state the optimality of the constant $c_o$ we suppose that
$c>c_o$.

Let us fix $i$ and consider the function $\varphi=|x-a_i|^\gamma$,  $\gamma\in (1-\frac{N}{2},0)$.
The function  $\varphi$ belongs to $H^1_\mu$ and 

\begin{equation*}
\int_{{\mathbb R}^N}\left(|\nabla \varphi |^2-c\,\frac{\varphi^2}{|x-a_i|^2}\right)\,d\mu
=(\gamma^2-c)\int_{{\mathbb R}^N}|x-a_i|^{2(\gamma -1)}\,d\mu.
\end{equation*}
Hence the bottom of the spectrum $\lambda_1$ of the operator $-(L+V)$ satisfies
\begin{equation}\label{lambda 1}
\lambda_1\le (\gamma^2-c)
\frac{\int_{{\mathbb R}^N}|x-a_i|^{2(\gamma -1)}\,d\mu}
{\int_{{\mathbb R}^N}|x-a_i|^{2\gamma}\,d\mu}
\end{equation}
since 
\begin{equation*}
\int_{{\R}^N}\left(|\nabla \varphi |^2-V\varphi^2\right)\,d\mu \le
 \int_{{\mathbb R}^N}\left(|\nabla \varphi |^2-c\,\frac{\varphi^2}{|x-a_i|^2}\right)\,d\mu.
\end{equation*}
We are able to state that for any $i\in \left\lbrace 1,\dots,n \right\rbrace $ it holds

\begin{equation}\label{pesi equivalenti}
C_1\,e^{-\alpha_2(2n-1)\frac{|x-a_i|^2}{2}}\le e^{-\sum_{i=1}^n\frac{|A^{\frac{1}{2}}(x-a_i)|^2}{2}}\le 
C_2\, e^{-\alpha_1 \frac{n+1}{2}\frac{|x-a_i|^2}{2}}
\end{equation}
\medskip
\noindent with $C_1=e^{-\alpha_2\sum_{i\ne j}|a_i-a_j|^2 } $ and $ C_2=
e^{{\frac{\alpha_1}{2}}\sum_{i\ne j}|a_i-a_j|^2}$
which is a consequence of the inequalities
$$\alpha_1 \sum_{i=1}^n|x-a_i|^2\le \sum_{i=1}^n|A^{\frac{1}{2}}(x-a_i)|^2\le \alpha_2\sum_{i=1}^n|x-a_i|^2,
\qquad \alpha_1\, ,\alpha_2>0,$$
and

\begin{equation}\label{equivalenza moduli}
\begin{split}
-\sum_{j\ne i}|a_i-a_j|^2+\frac{n+1}{2}|x-a_i|^2\le & \sum_{i=1}^n|x-a_i|^2
\\ &
\le(2n-1)|x-a_i|^2+2\sum_{j\ne i}|a_i-a_j|^2.
\end{split}
\end{equation}
The inequality (\ref{equivalenza moduli}) is proved in Appendix.

For simplicity in the following we place $\tilde \alpha_1=\alpha_1 \frac{n+1}{2}$ and 
$\tilde \alpha_2=\alpha_2(2n-1)$.

The equivalence between the weight functions in the case of one pole and in the case of multiple poles
allows us to calculate integrals in (\ref{lambda 1}). Indeed, by a change of variables and by
(\ref{pesi equivalenti})

\begin{equation}\label{stima a destra}
\begin{split}
\int_{\R^N}|x-a_i|^{2\beta}e^{-\sum_{i=1}^n\frac{|A^{\frac{1}{2}}(x-a_i)|^2}{2}}\,dx &\le 
C_2 \int_{\R^N}|x-a_i|^{2\beta}e^{-\tilde\alpha_1\frac{|x-a_i|^2}{2}}\,dx
\\&=
C_2\, 2^{\beta+\frac{N}{2}}\tilde\alpha_1^{-\beta-\frac{N}{2}}\int_{\R^N}|x-a_i|^{2\beta}e^{-\frac{|x-a_i|^2}{2}}\,dx.
\end{split}
\end{equation}
Taking in mind the definition of Gamma integral function 

$$\int_{\R^N}|x|^{2\beta}e^{-\frac{|x|^2}{2}}\,dx=
\sigma_N\, 2^{\beta+\frac{N}{2}-1}\Gamma\left(\beta+\frac{N}{2}\right),
\quad \beta+\frac{N}{2}>0,$$
we get from (\ref{stima a destra})

\begin{equation}\label{stima a destra n2}
\int_{\R^N}|x-a_i|^{2\beta}e^{-\sum_{i=1}^n\frac{|A^{\frac{1}{2}}(x-a_i)|^2}{2}}\,dx\le 
C_2\, 2^{2\beta+N-1} \tilde\alpha_1^{-\beta-\frac{N}{2}}\sigma_N \Gamma\left(\beta+\frac{N}{2}\right).
\end{equation}
Reasoning as above we obtain an estimate from below

\begin{equation}\label{stima a sinistra}
\begin{split}
\int_{\R^N}|x-a_i|^{2\beta}e^{-\sum_{i=1}^n\frac{|A^{\frac{1}{2}}(x-a_i)|^2}{2}}\,dx &
 \ge C_1 \int_{\R^N}|x-a_i|^{2\beta}e^{-\tilde\alpha_2\frac{|x-a_i|^2}{2}}\,dx
\\&=
C_1\tilde\alpha_2^{-\beta-\frac{N}{2}}\int_{\R^N}|x-a_i|^{2\beta}e^{-\frac{|x-a_i|^2}{2}}\,dx 
\\&=
C_1\, 2^{\beta+\frac{N}{2}-1} \tilde\alpha_2^{-\beta-\frac{N}{2}}\sigma_N \Gamma\left(\beta+\frac{N}{2}\right).
\end{split}
\end{equation}
Therefore, using (\ref{stima a destra n2}) and (\ref{stima a sinistra}), we get

\begin{equation*}
\begin{split}
\frac{\int_{\R^N}|x-a_i|^{2(\gamma-1)}\,d\mu}{\int_{\R^N}|x-a_i|^{2\gamma}\,d\mu}&\ge  
\frac{C_1\, 2^{\gamma+\frac{N}{2}-2} \tilde\alpha_2^{-\gamma-\frac{N}{2}+1}\sigma_N \Gamma(\gamma+
\frac{N}{2}-1)}{C_2\, 2^{2\gamma+N-1} \tilde\alpha_1^{-\gamma-\frac{N}{2}}\sigma_N \Gamma(\gamma+\frac{N}{2})}
\\ &=
\frac{C_1\, 2^{\gamma+\frac{N}{2}-2} \tilde\alpha_2^{-\gamma-\frac{N}{2}+1}}
{C_2\, 2^{2\gamma+N-1} \tilde\alpha_1^{-\gamma-\frac{N}{2}}(\gamma+\frac{N}{2}-1)}\,.
\end{split}
\end{equation*}
Then
$$\lambda_1\le 
\lim_{\gamma \rightarrow \left( 1- \frac{N}{2} \right)^+ } (\gamma^2 -c)
\frac{C_1\, 2^{\gamma+\frac{N}{2}-2} \tilde\alpha_2^{-\gamma-\frac{N}{2}+1}}
{C_2\, 2^{2\gamma+N-1} \tilde\alpha_1^{-\gamma-\frac{N}{2}}(\gamma+\frac{N}{2}-1)}=-\infty.$$
Thus, for any $M>0$, there is $\varphi \in H^1_\mu$ such that
$$\int_{{\mathbb R}^N} |\nabla \varphi|^2\,d\mu -c
\int_{{\mathbb R}^N}\frac{\varphi^2}{|x-a_i|^2}\,d\mu <
-M\int_{{\mathbb R}^N}\varphi^2\,d\mu.$$
By taking $M:=\frac{k+(n+1)c}{ r_0^2}+\frac{n}{2}\Tr A$ we find
$\varphi \in H^1_\mu$ such that
$$c\int_{{\mathbb R}^N}\frac{\varphi^2}{|x-a_i|^2}\,d\mu >
\int_{{\mathbb R}^N} |\nabla
\varphi|^2\,d\mu+\left[\frac{k+(n+1)c}{ r_0^2}+\frac{n}{2}\Tr A \right]
\int_{{\mathbb R}^N}\varphi^2\,d\mu$$ 
 which leads to a contradiction with respect the weighted Hardy inequality 
(\ref{gaussian hardy}) because, of course,
\begin{equation*}
c\int_{{\mathbb R}^N}\frac{\varphi^2}{|x-a_i|^2}\,d\mu \le
c \int_{{\mathbb R}^N}\sum_{i=1}^n\frac{\varphi^2}{|x-a_i|^2}\,d\mu.
\end{equation*}
This proves the optimality of $c_o$.
\end{proof}
We remark that when $c\in (0,\frac{c_o}{n}]$ the constant on the right-hand side of (\ref{gaussian hardy})
can be improved using a different proof based on the multipolar Hardy inequality in the case of Lebesgue measure. 
Moreover the inequality (\ref{c_o su n}) below holds also in the case $n=1$.

\begin{thm}\label{improved constant}
Assume $N\ge 3$ and $n\ge 1$. Then we get

\begin{equation}\label{c_o su n}
\frac{c_o}{n}\int_{{\R}^N}\sum_{i=1}^n\frac{\varphi^2 }{|x-a_i|^2}\, d\mu\le  
\int_{{\R}^N} |\nabla\varphi|^2 \, d\mu +
\frac{n}{2}\Tr A \int_{\R^N}\varphi^2 \, d\mu 
\end{equation}
for any $\varphi \in H^1_\mu(\R^N)$, where  $c_o=c_o(N):=\left(\frac{N-2}{2}\right)^2$.
\end{thm}

\begin{proof}
We start from the known inequality

\begin{equation}\label{lebesgue pole}
\frac{c_o}{n}\int_{\R^N}\sum_{i=1}^n\frac{\varphi^2 }{|x-a_i|^2}\, dx \le  
\int_{{\R}^N} |\nabla\varphi|^2 \, dx
\end{equation}
for all $\varphi \in H^1(\R^N)$, where $c_o=c_o(N):=\left(\frac{N-2}{2}\right)^2$,
which we can get  immediately by using the Hardy inequality with one pole.

% Riferimento da inserire poi

Then we apply the inequality (\ref{lebesgue pole})
to the function  $\varphi \sqrt{\mu}$ and reason as in the proof of Theorem \ref{wH}.
\end{proof}

\bigskip\bigskip

\section{Proof of the weighted Hardy inequality via the IMS method}\label{via IMS}

We can prove the inequality in Theorem \ref{wH} using the so-called 
IMS method, which consists 
in localizing the wave functions around the singularities by using a partition of unity.

We say that a finite family
$\left\lbrace J_i \right\rbrace_{i=1}^{n+1} $ 
of real valued functions
$J_i\in W^{1,\infty}(\R^N)$ 
is a \textit{partition of unity} in $\R^N$ if $\sum_{i=1}^{n+1}J_i^2=1$.\\ 
Any family of this type has the following properties:
\begin{enumerate}
\item[(a)] $\sum_{i=1}^{n+1}J_i\partial_\alpha J_i=0$ for any $\alpha= 1,\dots,N $;
\item[(b)] $J_{n+1}=\sqrt{1-\sum_{i=1}^{n}J_i^2}$;
\item[(c)] $\sum_{i=1}^{n+1}|\nabla J_i|^2\in L^{\infty}(\R^N)$.
\end{enumerate}
Furthermore we require that
\begin{equation}\label{omega}
\Omega_i\cap\Omega_j=\emptyset \quad\text{for any } i,j= 1,\dots,n ,\, i\neq j,
\end{equation}
where $\overline{\Omega}_i={\rm supp}(J_i)$, $ i=  1,\dots,n $. By the property (a) we get
$$\sum_{\alpha =1}^{N}|J_{n+1}\partial_\alpha J_{n+1}|^2=\sum_{\alpha =1}^{N}\left|
\sum_{j =1}^{n}J_j\partial_\alpha J_j \right|^2 =\sum_{\alpha =1}^{N}\sum_{j =1}^{n}|J_j\partial_\alpha J_j|^2 ,$$
from which 
$$|\nabla J_{n+1}|^2=\sum_{i=1}^{n}\frac{J_i^2}{1-J_i^2}|\nabla J_i|^2.$$ 
As a consequence we obtain an explicit formula for the sum of the gradients:
\begin{enumerate}
\item[(d)] $\sum_{i=1}^{n+1}|\nabla J_i|^2= \sum_{i=1}^{n}|\nabla J_i|^2+\sum_{i=1}^{n}\frac{J_i^2}{1-J_i^2}|
\nabla J_i|^2=\sum_{i=1}^{n}\frac{|\nabla J_i|^2}{1-J_i^2}$,
\end{enumerate}
Note that to avoid a singularity for the gradient of $J_{n+1}$ at the points where $1-J_i^2=0$, from (d) we shall assume
the additional constraint $|\nabla J_i|^2=F(x)(1-J_i^2)$, for $i= 1,\dots,n $ and for some $F\in L^{\infty}(\R^N)$.

By proceeding as in \cite[Lemma 2]{BDE}, 
%originally established by B. Simon in \cite{Simon}, 
we are able to state the following result.
\begin{lemma}\label{lemma2BDE}
Let $\left\lbrace J_i \right\rbrace_{i=1}^{n+1} $ be a partition of unity satisfying (\ref{omega}), and $d\mu$ the Gaussian 
measure defined in (\ref{mu}) . For any $u\in H^1_\mu$ and any $V\in L^1_{loc}(\R^N)$ we get
\begin{equation*}
\begin{split}
\int_{\R^N}\left( |\nabla \varphi|^2 -V\varphi^2 \right)d\mu = &
 \sum_{i=1}^{n+1}\int_{\R^N}( |\nabla (J_i \varphi)|^2 - V(J_i \varphi)^2 ) d\mu\\
 & -\int_{\R^N}\sum_{i=1}^{n+1} |\nabla J_i|^2\varphi^2\,d\mu.
\end{split}
\end{equation*}

\end{lemma}
\begin{proof}
We can immediately observe that
\begin{equation}\label{primo pezzo}
\int_{\R^N} V \left( \sum_{i=1}^{n+1}(J_i\varphi)^2\right)\,d\mu=\int_{\R^N} V \left( \sum_{i=1}^{n+1}J_i^2\right)
\varphi^2\,d\mu=\int_{\R^N} V\varphi^2\,d\mu.
\end{equation}
On the other hand,
\begin{equation}\label{calcoloNabla}
\begin{split}
\sum_{i=1}^{n+1}|\nabla \left(J_i\varphi \right)|^2 &=\sum_{i=1}^{n+1}|(\nabla J_i)\varphi + (\nabla \varphi)J_i|^2\\
&=\sum_{i=1}^{n+1}|\nabla J_i|^2\varphi^2+\sum_{i=1}^{n+1}|\nabla \varphi|^2J_i^2+2\sum_{i=1}^{n+1}(J_i\nabla J_i) 
(\varphi \nabla \varphi) \\
&=\sum_{i=1}^{n+1}|\nabla J_i|^2\varphi^2+|\nabla \varphi|^2+\left( \sum_{i=1}^{n+1}J_i\nabla J_i \right) \nabla \varphi^2.
\end{split}
\end{equation}
By property (a) it follows that $\left( \sum_{i=1}^{n+1}J_i\nabla J_i \right) \nabla \varphi^2=0$, then by integrating 
(\ref{calcoloNabla}) on $\R^N$ we obtain
\begin{equation}\label{secondo pezzo}
\int_{\R^N}|\nabla \varphi|^2\,d\mu=\int_{\R^N}\sum_{i=1}^{n+1}|\nabla \left(J_i\varphi \right)|^2\,d\mu-\int_{\R^N} 
\sum_{i=1}^{n+1} |\nabla J_i|^2 \varphi^2\,d\mu.
\end{equation}
From (\ref{primo pezzo}) and (\ref{secondo pezzo}) we get the result.
\end{proof}

Taking in mind that 
$$V_n(x)=\sum_{i=1}^{n}\frac{1}{|x-a_i|^2},$$ as defined in (\ref{Vn}), we recall a preliminary lemma, stated by 
Bosi, Dolbeault and Esteban in \cite{BDE}, about the case $n=2$, with $a_1=a$, $a_2=-a$ and $0<r_0\le |a|$.
\begin{lemma}\label{lemma3BDE}
There is a partition of the unity $\left\lbrace J_i \right\rbrace_{i=1}^3 $ satisfying (\ref{omega}) with $J_1\equiv 1$ on 
$B(a,\frac{r_0}{2})$, $J_1\equiv 0$ on $B(a,r_0)^c$, $J_2(x)=J_1(-x)$ for any $x\in \R^N$, $0<r_0\le |a|$, such that, for any $c>0$,
 there exists a constant $k\in \left[ 0, \pi^2 \right) $ for which, almost everywhere for all $x\in \Omega:=
{\rm supp}(J_1)\cup{\rm supp}(J_2)$, we have 
\begin{equation}
\sum_{i=1}^{3}|\nabla J_i|^2 + c\,J_3^2\,V_2(x)=\sum_{i=1,2} \frac{|\nabla J_i|^2}{1-J_i^2}+ 
c\,J_3^2\,V_2(x)\le \frac{k+2c}{r_0^2}.
\end{equation}
\end{lemma}

\medskip

Now we are able to proceed with the proof.
\begin{proof}[Proof of Theorem \ref{wH}]
Let us define the following quadratic form

\begin{equation}\label{formaQ}
\begin{split}
Q[\varphi]:=\int_{\R^N}\left( |\nabla \varphi|^2 -cV_n(x)\varphi^2 \right)d\mu, \qquad \varphi\in H^1_\mu.
\end{split}
\end{equation}
By virtue of Lemma \ref{lemma2BDE} we are able to write (\ref{formaQ}) as follows
\begin{equation}\label{formaQ2}
Q[\varphi]=\sum_{i=1}^{n} Q[J_i\varphi] + R_n, \qquad \varphi\in H^1_\mu
\end{equation}
where
$$R_n =\int_{\R^N}|\nabla (J_{n+1}\varphi)|^2\, d\mu -c\int_{\R^N} V_n|J_{n+1}\varphi|^2 \,d\mu-\sum_{i=1}^{n+1}
\int_{\R^N}|\nabla J_i|^2\varphi^2\,d\mu.$$
Thanks to the property (d) we have
\begin{equation*}
\begin{split}
R_n &= \int_{\R^N}|\nabla (J_{n+1}\varphi)|^2\, d\mu -c\int_{\R^N} V_n\left(1-\sum_{i=1}^{n}J_i^2 \right)\varphi^2 \,d\mu-
\sum_{i=1}^{n}\int_{\R^N}\frac{|\nabla J_i|^2}{1-J_i^2}\varphi^2\,d\mu\\
&\ge -c\int_{\R^N}V_n(x)\left( 1-\sum_{i=1}^{n}J_i^2\right) \varphi^2\,d\mu -\sum_{i=1}^{n}
\int_{\R^N}\frac{|\nabla J_i|^2}{1-J_i^2}\varphi^2\,d\mu.
\end{split}
\end{equation*}
Let us consider a partition of unity $\left\lbrace J_i \right\rbrace_{i=1}^{n+1}$ satisfying (\ref{omega}), and the sets 
$\Omega_i=B(a_i,r_0)$ such that $\overline{\Omega}_i={\rm supp}(J_i)$, $i= 1,\dots,n $. If we set $\Omega=
\cup_{i=1}^n \overline{\Omega}_i$ and $\Gamma =\R^N \setminus \Omega$, then $|x-a_i|\ge r_0$ in $\overline{\Omega}_j$ for
 $i\neq j$, and $V_n(x)\le\frac{n}{r_0^2}$ on $\Gamma$. \\
Moreover, using the condition (\ref{omega}) we get
\begin{equation*}
R_n \ge-\sum_{i=1}^{n}\int_{\Omega_i}\left[ \frac{|\nabla J_i|^2}{1-J_i^2}+c\left(1-J_i^2\right)V_n(x)\right]\varphi^2 \,d\mu-
\frac{c\,n}{r_0^2}\int_{\Gamma}\varphi^2\,d\mu. 
\end{equation*}
Taking into account that $J_j=0$ on $\Omega_i$ for any $j\neq i$,  we have for $j\neq i$
\begin{equation*}
\begin{split}
R_n \ge&-\sum_{i=1}^{n}\int_{\Omega_i}\Biggl[ \frac{|\nabla J_i|^2}{1-J_i^2}+
\frac{|\nabla J_j|^2}{1-J_j^2}+ c\left(1-J_i^2-J_j^2\right)
\left(\frac{1}{|x-a_i|^2}+\frac{1}{|x-a_j|^2} \right) 
\\&
+\biggl. c\left(1-J_i^2\right)\left( \sum_{k\ne i,j}\frac{1}{|x-a_k|^2}\right) \Biggr]\varphi^2 \,d\mu-
\frac{c\,n}{r_0^2}\int_{\Gamma}\varphi^2\,d\mu,
\end{split} 
\end{equation*}
%(We can obtain this choosing e.g. $j=i+1$ for $i=1,\dots,n-1$, and $j=1$ for $i=n$.)\\
Now, taking $\left\lbrace J_i, J_j, \sqrt{1-J_i^2-J_j^2} \right\rbrace $ as  the partition of unity, we can apply Lemma \ref{lemma3BDE} 
on $\Omega_i$ with $(a_i,a_j)=(-a,a)$ up to a change of coordinates. In this way we get
\begin{equation}
\begin{split}\label{Rn}
R_n\ge& -\sum_{i=1}^{n} \int_{\Omega_i} \left[ \frac{k+2c}{r_0^2}+c(1-J_i^2)\left( \sum_{k\ne i,j}\frac{1}{|x-a_k|^2} \right)  \right] 
\varphi^2 \, d\mu -\frac{c\,n}{r_0^2}\int_{\Gamma}\varphi^2 \, d\mu\\
\ge &-\sum_{i=1}^{n} \int_{\Omega_i} \left[ \frac{k+2c}{r_0^2}+\frac{(n-2)c}{r_0^2}(1-J_i^2) \right] \varphi^2 \, d\mu -
\frac{c\,n}{r_0^2}\int_{\Gamma}\varphi^2 \, d\mu,
\end{split}
\end{equation}
since we can estimate $\frac{1}{|x-a_k|^2}$ by $\frac{1}{r_0^2}$ for all $k\neq i,j$.
Taking into account (\ref{formaQ}) and using the weighted Hardy inequality (\ref{lebesgue pole}) with $n=1$ we get
\begin{equation*}
\begin{split}
Q[J_i \varphi]=&\int_{\R^N}|\nabla J_i\varphi|^2\, d\mu - c\int_{\R^N}\Biggl( \frac{1}{|x-a_i|^2}+
\sum_{\substack{j=1\\j\neq i}}^{n}\frac{1}{|x-a_j|^2} \Biggr)|J_i\varphi|^2\, d\mu \\
\ge& -\left[ \frac{1}{2}{\rm Tr} A+\frac{(n-1)c}{r_0^2}\right] \int_{\Omega_i} |J_i\varphi|^2 \, d\mu, 
\end{split}
\end{equation*}
from which
\begin{equation}\label{sommaQj}
\sum_{i=1}^{n}Q[J_i \varphi] \ge - \frac{1}{2}{\rm Tr} A\sum_{i=1}^{n}\int_{\Omega_i} \varphi^2 \, d\mu -
\frac{(n-1)c}{r_0^2}\sum_{i=1}^{n} \int_{\Omega_i} J_i^2\varphi^2 \, d\mu 
\end{equation}
From (\ref{formaQ2}), (\ref{Rn}) and (\ref{sommaQj}) we deduce
\begin{equation*}
\begin{split}
Q[\varphi]\ge &- \sum_{i=1}^{n}\int_{\Omega_i}\left[\frac{k+2c}{ r_0^2}+\frac{(n-2)c}{ r_0^2}(1-J_i^2)+
\frac{1}{2}{\rm Tr} A+\frac{(n-1)c}{ r_0^2}J_i^2\right] \varphi^2 \, d\mu\\
& - \frac{c\,n}{r_0^2}\int_{\Gamma}\varphi^2 \, d\mu.
\end{split}
\end{equation*}
Since
\begin{equation*}
k+2c+c(n-2)(1-J_i^2)+c(n-1)J_i^2 = k+cn+cJ_i^2\le k+c(n+1),
\end{equation*}
we finally obtain
\begin{equation*}
\begin{split}
Q[\varphi]\ge & -\left[\frac{k+(n+1)c}{ r_0^2}+\frac{1}{2}{\rm Tr} A\right] \int_{\Omega}\varphi^2 \, d\mu - 
\frac{c\,n}{r_0^2}\int_{\Gamma}\varphi^2 \, d\mu\\
\ge& -\left[\frac{k+(n+1)c}{ r_0^2}+\frac{1}{2} {\rm Tr} A\right] \int_{\R^N}\varphi^2 \, d\mu ,
\end{split}
\end{equation*}
\end{proof}

\bigskip\bigskip

\section{ Existence of solutions via weighted Hardy inequality}

The potential $ V(x)=\sum_{i=1}^n \frac{c}{|x-a_i|^2}$ 
and the Gaussian density $\mu(x)$ satisfy the hypotheses of the Theorem \ref{theor as CM}.
We can therefore state the following existence and nonexistence result
as a consequence of the weighted Hardy inequality (\ref{gaussian hardy}) and of the Theorem \ref{theor as CM}.

\begin{thm}\label{Hardy e CM theor}
Assume that $N\geq 3$, $A$ a positive definite real Hermitian $N\times N$-matrix and $0\leq V(x)\leq
 \sum_{i=1}^n\frac{c}{|x-a_i|^2}$, with $c> 0$, $x, a_i \in \R^N$, $i \in \lbrace 1,\dots, n\rbrace$. 
Let $L$ the Ornstein-Uhlenbeck type operator (\ref{OU}). Then the following assertions hold:
\begin{itemize}
\item[i)] If $c \leq c_o$ there exists a positive weak solution $u\in C\left(\left[ 0,\infty\right), L^2_\mu\right)$ of
\begin{equation}\label{problemA}
\left\lbrace
\begin{array}{ll}
\partial_t u(x,t)=L+V(x)u(x,t),\quad x\in \R^N, t>0,
\\
u(\cdot,t)=u_0\in L^2_\mu,
\end{array}\right.
\end{equation}
satisfying 
\begin{equation}\label{stima2}
\|u(t)\|_{L^2_\mu}\leq M e^{\omega t}\|u_0\|_{L^2_\mu}, \qquad t\geq 0
\end{equation}
for some constants $M\geq1$, $\omega\in \R$, and any $u_0\in L^2_\mu$.
\item[ii)] If $c > c_o$ there exists no positive weak solution of (\ref{problemA}) with $V(x)=
\sum_{i=1}^n\frac{c}{|x-a_i|^2}$ satisfying (\ref{stima2}) for any $0\leq u_0\in L^2_\mu$, $u_0\neq 0$.
\end{itemize}
\end{thm}
\bigskip
%Nota: a_c e' settoriale perchè bilineare (non sesquilineare) quasi-accretiva, 
%dunque l'op. associato A e' settoriale, il semigruppo e' analitico.

Following a different approach based on bilinear forms associated to the operator $-(L+V)$, 
we obtain an existence result.
We state the generation of an analytic $C_0$-semigroup. 

Let us define the bilinear form 
\begin{equation}\label{a_c}
a_c(u,v):=\int_{\R^N}\nabla u \cdot \nabla v \, d\mu - 
c\sum_{i=1}^n \int_{\R^N} \frac{uv}{|x-a_i|^2} \, d\mu 
\end{equation}
for $u,v\in D(a_c)=H^1_\mu$, $N\geq 3$ and $c>0$. 

Arguing as in \cite[Propositions 2.2 and 2.3]{ArendtGoldstein}, we can get the next result. 

\begin{prop}
The following statements hold:
\begin{itemize}
\item[ i)] $a_c$ is closed if $c < c_o$;
\item[ ii)] $a_{c_o}$ is closable; 
\end{itemize}
\end{prop}

Furthermore $a_c$ is quasi-accretive for all $c\in \left( 0, c_o \right]$. 
In fact by the weighted Hardy inequality (\ref{gaussian hardy}) we immediately get
$$a_c(u,u)\ge -K\left(u,u \right)_{H^1_\mu}$$
for all $u\in H^1_\mu$, with $K$ the constant on the right-hand side in the inequality.

Then, for $c < c_o$, the associated operator $\mathcal{A}$ on $L^2_\mu$  defined by
\begin{equation}
D(\mathcal{A})=\left\lbrace
\begin{split}
u \in D(a_c) : \exists\, v \in L^2_\mu \quad\textit{s. t.} \quad   
  a_c(u,\phi)=\int_{\R^N} v\phi \, d\mu \quad \forall \phi \in D(a_c)   
\end{split} 
\right\rbrace ,
\notag 
\end{equation}
$$ \mathcal{A}u=v .$$
Then $-\mathcal{A}=L+V$ generates an analytic $C_0$-semigroup $\lbrace S(t)
\rbrace_{t\ge 0}$ on $L^2_\mu$ satisfying
$$\|S(t)\|\leq e^{K t}, \quad t\geq 0.$$
For the case $c = c_o$ the same conclusion holds taking the closure $\overline{a_{c_o}}$ 
instead of $a_{c_o}$ in the definition of $\mathcal{A}$.

The positivity of the solution $u$ can be obtained as in \cite[Section 2]{ArendtGoldstein}.
Indeed, we can regard $S(t)$ as the limit of positive preserving semigroups described by cut-off potentials.

Let $\mathcal{A}_k=L+\min\left( V, c k\right) $, $k\in \mathbb N$.
Since $L$ is the generator of a positive preserving semigroup on $L^2_\mu$ and
 $\min\left( V, k\right)$ is bounded and non-negative, $\mathcal{A}_k$ 
generates a positive preserving semigroup, denoted by $S_k(t)$. Moreover
$$0\le S_k(t)\le S_{k+1}(t).$$
% e' una disuguaglianza tra operatori esponenziali
If $c\le c_o$ it follows from the monotone convergence theorem for forms (cf \cite[Theorem S.14]{Reed})
that
$$\lim\limits_{k\to \infty}S_k(t)=S(t)$$
strongly in $L^2_\mu$. Then $u(t)=S(t)u_0$ is positive.

Finally, as in  \cite[Proposition 2.5]{ArendtGoldstein}, we can observe that if $c>c_o$ then 
$$\lim\limits_{k\to \infty} \|S_k(t)\|=\infty, \quad t>0.$$

\bigskip\bigskip\bigskip

\section*{Appendix}\label{dim equivalenza moduli}
%\bigskip

Let us state the following estimates
\begin{equation}
\begin{split}
-\sum_{j\ne i}|a_i-a_j|^2+\frac{n+1}{2}|x-a_i|^2\le & \sum_{i=1}^n|x-a_i|^2
\\ &
\le(2n-1)|x-a_i|^2+2\sum_{j\ne i}|a_i-a_j|^2
\end{split}
\end{equation}
for any $i,j\in\{1,\dots, n\}$.

In fact 
$$|x-a_j|^2=|x-a_i+a_i-a_j|^2\le 2|x-a_i|^2+2|a_i-a_j|^2$$
and
$$|x-a_j|^2\ge \frac{|x-a_i|^2}{2}-|a_i-a_j|^2.$$
As a consequence we obtain
$$\sum_{i=1}^n|x-a_i|^2=|x-a_i|^2+\sum_{j\ne i}|x-a_j|^2\le
 |x-a_i|^2+2(n-1)|x-a_i|^2+2\sum_{i\ne j}^n|a_i-a_j|^2$$
and 
$$\sum_{i=1}^n|x-a_i|^2\ge |x-a_i|^2+\frac{n-1}{2}|x-a_i|^2-\sum_{i\ne j}^n|a_i-a_j|^2.$$

\end{document}